\documentclass[sn-mathphys-num]{sn-jnl}

\usepackage{graphicx}%
\usepackage{multirow}%
\usepackage{amsmath,amssymb,amsfonts}%
\usepackage{amsthm}%
\usepackage{mathrsfs}%
\usepackage[title]{appendix}%
\usepackage{xcolor}%
\usepackage{textcomp}%
\usepackage{manyfoot}%
\usepackage{booktabs}%
\usepackage{algorithm}%
\usepackage{algorithmicx}%
\usepackage{algpseudocode}%
\usepackage{listings}%
\newcommand{\R}{{\mathbb{R}}}                
\newcommand{\N}{{\mathbb{N}}}                
\newcommand{\To}{\;{\longrightarrow}\;}      
\newcommand{{\X}}{\mathbb{X}}
\newcommand{{\Y}}{\mathbb{Y}}
\newcommand{{\Z}}{\mathbb{Z}}
\newcommand{{\W}}{\mathbb{W}}
\newcommand{\pa}[1]{\left(#1\right)}           
\newcommand{\set}[1]{\left\{#1\right\}}        
\newcommand{\abs}[1]{\left\vert#1\right\vert}
\newcommand{\norm}[1]{\left\Vert#1\right\Vert}

\newtheorem{thm}{Theorem}[section]
\newtheorem{lem}[thm]{Lemma}
\newtheorem{prop}[thm]{Proposition}
\newtheorem{defi}{Definition}[section]
\newtheorem{rem}{Remark}[section]
\theoremstyle{remark}
\newtheorem{example}{Example}[section]

\begin{document}

\title[Compact almost automorphic dynamics of non-autonomous differential equations with exponential dichotomy and applications to biological models with delay]{Compact almost automorphic dynamics of non-autonomous differential equations with exponential dichotomy and applications to biological models with delay}


\author*[1]{\fnm{Alan} \sur{Ch\'avez}}\email{ajchavez@unitru.edu.pe}

\author[2]{\fnm{Nelson} \sur{Aragon\'es}}\email{naragones@unitru.edu.pe}
\equalcont{These authors contributed equally to this work.}

\author[3]{\fnm{Manuel} \sur{Pinto}}\email{pintoj.uchile@gmail.com}
\equalcont{These authors contributed equally to this work.}

\author[4]{\fnm{Ulices} \sur{Zavaleta}}\email{auzavaleta@unitru.edu.pe}
\equalcont{These authors contributed equally to this work.}

\affil*[1]{\orgdiv{OASIS and GRACOCC research groups, Instituto de investigaci\'on en Matem\'aticas}, \orgname{Facultad de Ciencias F\'isicas y Matem\'aticas, Universidad Nacional de Trujillo}, \orgaddress{\street{Av. Juan Pablo II S/N}, \city{Trujillo}, \postcode{13011}, \state{La Libertad}, \country{Per\'u}}}

\affil[2,4]{\orgdiv{OASIS research group, Instituto de investigaci\'on en Matem\'aticas}, \orgname{Facultad de Ciencias F\'isicas y Matem\'aticas, Universidad Nacional de Trujillo}, \orgaddress{\street{Av. Juan Pablo II S/N}, \city{Trujillo}, \postcode{13011}, \state{La Libertad}, \country{Per\'u}}}

\affil[3]{\orgdiv{Departamento de Matem\a'aticas, Facultad de Ciencias}, \orgname{Universidad de Chile}, \orgaddress{\street{Las Palmeras 3425}, \city{Santiago}, \postcode{7800003}, \state{Regi\'on Metropolitana}, \country{Chile}}}


\abstract{
In the present work, we prove that, if $A(\cdot)$ is a compact almost automorphic matrix and the system
$$x'(t) = A(t)x(t)\, ,$$
possesses an exponential dichotomy with Green function $G(\cdot, \cdot)$, then its associated system
$$y'(t) = B(t)y(t)\, ,$$
where $B(\cdot) \in H(A)$ (the hull of $A(\cdot)$) also possesses an exponential dichotomy. Moreover, the
Green function $G(\cdot, \cdot)$ is compact Bi-almost automorphic in $\mathbb{R}^2$, this implies that $G(\cdot, \cdot)$ is  $\Delta_2$ - like uniformly continuous, where $\Delta_2$ is the principal diagonal of $\mathbb{R}^2$, an important ingredient in the proof of invariance of the compact almost automorphic function space under convolution product with kernel $G(\cdot, \cdot)$. Finally, we study the existence of a positive compact almost automorphic solution of non-autonomous differential equations of biological interest having non-linear harvesting terms and mixed delays.
}




\maketitle

\section{Introduction}\label{sec1}

Let us consider the non-autonomous differential equation
\begin{equation}
x'(t) =A(t)x(t) + f(t)\, ,\label{eq1}
\end{equation}
and suppose that the associate system
\begin{equation}\label{eq3}
x'(t)= A(t)x(t)
\end{equation}
possesses an exponential dichotomy, then the unique bounded solution to equation (\ref{eq1}) is
\begin{equation}\label{SolDeff}
x(t)=\int_{\mathbb{R}}G (t,s)f(s) ds\, ,
\end{equation}
where $G(\cdot, \cdot)$ is the Green function associated to (\ref{eq3}). If  $A(\cdot)$ is an almost
 automorphic matrix and $f$ is an almost automorphic function, then the solution $x$ is almost automorphic if $G(\cdot,\cdot)$ is
 Bi-almost automorphic. To prove that $G(\cdot,\cdot)$ is Bi-almost automorphic is not trivial. For example, in infinite dimension, in order to do that, it is important to add as hypothesis that the resolvent operator also be almost automorphic, moreover it is necessary a delicate analysis of the behaviour of the Yosida approximations of the evolution family $A(\cdot)$, see for instance  \cite{baroun2008almost,baroun2019almost,esSe2022compact,maniar2003almost}. In contrasts, in the finite dimensional case, A. Coronel and co-authors have reported in \cite{coronel2016almost} the notion of integrable Bi-almost automorphy and proved that $G(\cdot,\cdot)$ is integrable Bi-almost automorphy if a projection matrix $P$ commutes with the fundamental matrix solution of  (\ref{eq3}).

The Bi-almost periodicity/automorphy notion of a continuous function of two variables, goes back at least to the work of T. J. Xiao et. al. \cite{19}. This notion is very useful in the study of almost periodic/automorphic solutions of non-autonomous differential equations (see the works  \cite{baroun2008almost,baroun2019almost,esSe2022compact,pinto2017pseudo} and references cited therein) and also of integral equations with kernel $\Gamma(t,s)$ depending on two variables, see for instance \cite{chavez2021almostaut}.

The main purpose of the present work is to prove, in finite dimension, that if $A(\cdot)$ is a compact almost  automorphic matrix, and
%
%
%
%
the system (\ref{eq3}) possesses  an $(\alpha, K, P)$-exponential dichotomy (Definition \ref{Green}), then the system
\begin{equation}\label{eq3h}
y'(t)= B(t)y(t)\, ,
\end{equation}
possesses an $(\alpha, K, P_0)$-exponential dichotomy, where $B(\cdot)\in H(A)$ and $H(A)$ stands for the hull of $A(\cdot)$ (see comments after Definition \ref{KaaDDef}), this is the analogous result to the case of almost periodic systems, see \cite[Theorem 7.6]{fink1974almost}. We also prove, without the techniques and the machinary of the infinite dimensional case, that the associated Green function $G(\cdot,\cdot)$ is compact Bi-almost automorphic  (Theorem \ref{teo001} in section 3).
Notice that, since $G(\cdot, \cdot)$ is discontinuous in the principal diagonal of $\mathbb{R}^2$, thus continuous almost everywhere, then it is measurable and therefore $G(\cdot, \cdot)$ does not fall into the notion of Bi-almost periodicity/automorphy introduced in \cite{19} which is given for continuous functions. This motivates the introduction, in this work, of the notion of (compact) Bi-almost automorphy for measurable functions. Also, we are able to prove that, although $G(\cdot, \cdot)$ is not uniformly continuous (Proposition \ref{CharCAA}), it is \textbf{$\Delta_2$-like uniformly continuous} (Lemma \ref{almostUC}), where $\Delta_2 :=\{ (t,t)\in \mathbb{R}^2 \, : \, t \in \mathbb{R}\} \subset \mathbb{R}^2$ denotes the principal diagonal of $\mathbb{R}^2$.

Also we study the existence of a positive compact almost automorphic solution of the following  biological model with non-linear harvesting terms and mixed delays 
\begin{equation}\label{eq11-ps-1}
\dot{u}(t)=
-\alpha(t)u(t)+\sum_{i=1}^n\beta_i(t)f_i(\lambda_i(t)u(t-\tau_i(t))+b(t)H(u(t-\sigma(t)))\, .
\end{equation}

The present work is organized as follows: in section \ref{section2}, we introduce the notion of compact Bi-almost automorphy for measurable functions and revisit some preliminary facts on compact almost automorphic functions. In section \ref{section3} we prove the main result of the present work, thus Theorem \ref{teo001} and Lemma \ref{almostUC}. In section \ref{section4}, we use the $\Delta_2$-like uniform continuity of $G(\cdot, \cdot)$ to show invariance of the almost automorphic function space under convolutions whose kernel is the Green function $G(\cdot, \cdot)$. Finally, in section \ref{section5} we prove the existence of a unique positive compact almost automorphic solution of the delay differential equation (\ref{eq11-ps-1}). 


\section{Preliminaries}\label{section2}
In the present work, $\mathbb{X}$ is a real or complex Banach spaces with norm $||\cdot||_{\mathbb{X}}$
, $||\cdot||$ denotes the norm of matrices,  $||\cdot||_{\infty}$ denotes the supremum norm and $|\cdot|$ is the absolute value.
\begin{defi}\label{KaaDDef}
A continuous function $f:\R\to\X$, is compact almost
automorphic if, for every sequence $\set{s'_n}\subset \mathbb{R}$, there
exists a subsequence $\set{s_n}\subseteq \set{s'_n}$ and a function
$\tilde{f}: \R\to\X$ such that the following limits hold
\begin{equation}\label{Kaadef}
\lim_{n\to + \infty}\sup_{t\in \mathcal{C}}||f(t+s_n)-\tilde{f}(t)||_{\mathbb{X}}=0,\quad \lim_{n\to +
\infty}\sup_{t\in \mathcal{C}}||\tilde{f}(t-s_n)-f(t)||_{\mathbb{X}}=0\, .
\end{equation}
where $\mathcal{C}\subset \mathbb{R}$, is compact.
\end{defi}
If in (\ref{Kaadef}) the limits are pointwise in $\mathbb{R}$, $f$ is called almost automorphic.
We denote by $\mathcal{K}AA(\R, \X)$ the apace of compact almost automorphic functions and by $AA(\R,\X)$ the space of almost automorphic functions. 
$AA(\R,\X)$ and $\mathcal{K}AA(\R,\X)$ are Banach spaces under the supremum norm and $\mathcal{K}AA(\R,\X)\subset AA(\R,\X)$ holds, see \cite{CHAKPPINTO2023}.

For an almost automorphic function $f$, its Hull, denoted by $H(f)$, is the set of all functions $\tilde{f}$ that satisfy Definition \ref{KaaDDef}.


The following is a useful characterization of compact almost automorphic functions on the real line. For the multidimensional euclidean space, see \cite{CHAKPPINTO2023}.

\begin{prop}\label{CharCAA} A continuous function $f:\mathbb{R}\to \mathbb{X}$ is compact almost automorphic if and only if it is almost automorphic and uniformly continuous.
\end{prop}

In the following proposition, we summarize some properties of compact almost automorphic function, for further details see for instance \cite{chavez2022multi,CHAKPPINTO2023} and the books \cite{diaganaBook2013almost,16}.

\begin{prop}\label{PropCAA}
Let $f,g \in \mathcal{K} AA (\mathbb{R}; \mathbb{X})$, then
\begin{enumerate}
\item If $\mathbb{X}$ be a Banach algebra with norm $||\cdot ||_{\mathbb{X}}$,  addition $+_{\mathbb{X}}$ and multiplication $\times_{\mathbb{X}}$, then $\mathcal{K} AA(\mathbb{R},\mathbb{X})$ is also a Banach algebra with norm $||\cdot ||_{\infty}$ and  the operations: if $f,g \in \mathcal{K} AA(\mathbb{R},\mathbb{X})$, then
$$ (f+ g)(t):=f(t) +_\mathbb{X} g(t)\, , \, \, \, t \in \mathbb{R}\, ,$$
$$ (f\cdot g)(t):=f(t) \times_\mathbb{X} g(t)\, , \, \, \, t \in \mathbb{R}\, .$$

\item $f$ is bounded  and, if $\tilde{f}$ is the limit function in definition \ref{KaaDDef}, then
$$||f||_{\infty}=||\tilde{f}||_{\infty}\, .$$

\item If $F: \mathbb{X} \to \mathbb{Y}$ be a continuous function, then $F\circ f:\mathbb{R} \to\mathbb{Y}$ is compact almost automorphic.

\item The range of $f$ is relatively compact.

\item If $a\in \mathcal{K} AA (\mathbb{R}; \mathbb{R}) $, then $F(t)=f(t-a(t))$ is compact almost automorphic from $\mathbb{R}$ to $\mathbb{X}$.

\end{enumerate}
\end{prop}
Next, we provide the definition of Bi-almost automorphy for measurable functions.

\begin{defi}\label{CompBiAAMeasu}
Let $f:\R\times\R \to\X$  be a measurable function, $f$ is compact Bi-almost automorphic if, for every sequence $\set{s'_n}_{n\in\N}\subset\R$ there exist a subsequence $\set{s_n}\subseteq \set{s'_n}$  and a function $\tilde{f}: \R\times\R \to\X$  such that, if $\mathbb{K} \subset \R\times\R$ is compact, the following limits hold
\begin{align*}
&\lim_{n\to + \infty}\sup_{(t,s)\in \mathbb{K}}||f(t+s_n, s+s_n)- \tilde{f}(t,s)||_{\mathbb{X}}=0\, , \\
& \lim_{n\to + \infty}\sup_{(t,s)\in \mathbb{K}}||\tilde{f}(t-s_n, s-s_n)-f(t,s)||_{\mathbb{X}}=0\, .
\end{align*}
When the limits are pointwise in $\R\times\R$, $f$ is called Bi-almost automorphic.
\end{defi}

\begin{defi}\label{Green}
Let $\Phi(\cdot )$ be a fundamental matrix of system  (\ref{eq3}), then (\ref{eq3}) has an exponential dichotomy with parameters $(\alpha, K, P)$, if there exist positive constants $\alpha, K$ and
a projection $P$ ($P^2=P$) such that $||G(t,s)||\leq
Ke^{-\alpha\abs{t-s}}$,\; $t,s\in\R$. Where $G(\cdot , \cdot)$ is the Green function defined by:
$$
G(t,s):=
\begin{cases}
\Phi(t)P\Phi^{-1}(s),\quad &\text{$t\geq s$\,, }\\
-\Phi(t)(I-P)\Phi^{-1}(s),\quad &\text{$t < s$}\, .
\end{cases}
$$
In this case, we say that system (\ref{eq3}) has an $(\alpha, K, P)$-exponential dichotomy.
\end{defi}
\begin{defi}\cite{chavez2021almostaut}\label{defBaa}
We say that a measurable function $C:\R\times\R  \to \X$ is $\lambda$-bounded if,  there exists a positive function $\lambda:\R\times\R\to \R$ such that, for every $\tau \in \R$ we have
$$||C(t+\tau,s+\tau )||_{\mathbb{X}} \leq \lambda(t,s)\, .$$
\end{defi}

\begin{example}If $G(\cdot,\cdot)$ is the Green function defined in Definition \ref{Green}; then, $G(\cdot,\cdot)$ is measurable and $\lambda$-bounded, for $\lambda(t,s)=Ke^{-\alpha|t-s|}$ .
\end{example}

\begin{lem}\label{lemBou}
Let us suppose that the Bi-almost automorphic function $C:\R\times\R \to \mathbb{X}$  is $\lambda$-bounded.
Then, its limit function $\tilde{C}:\R\times\R \to \mathbb{X}$ (see definition \ref{defBaa}) satisfies:
$$||\tilde{C}(t+\tau,s+\tau )||_{\mathbb{X}} \leq \lambda(t,s)\,\, , \, \, \forall \tau \in \mathbb{R}\, .$$
\end{lem}

\begin{proof}
Let $\{s_n'\}\subset \mathbb{R}$ be arbitrary. Since $C$ is Bi-almost automorphic, there exist a subsequence $\{s_n\} \subset \{s_n'\}$ and a function $\tilde{C}$ such that the following pointwise limits hold
%
%
$$\tilde C(t,s ):=\lim_{n\to +\infty}C(t+s_n,s+s_n )\, \, ,\, \, \,  C(t,s)=\lim_{n\to +\infty}\tilde C(t-s_n,s-s_n )\, .$$ 
On the other hand, we have
$$||\tilde{C}(t+\tau,s+\tau )||_{\mathbb{X}}\leq ||\tilde{C}(t+\tau,s+\tau )-C(t+\tau+s_n,s+\tau+s_n )||_{\mathbb{X}}+ \lambda(t,s)\,.$$
Now, taking the limit as $n\to +\infty$ in the last inequality 
we obtain the result.
\end{proof}
This Lemma improves \cite[Lemma 2.9]{chavez2021almostaut}, where it was proved that $\tilde{C}(\cdot, \cdot)$ satisfies the inequality $||\tilde{C}(t ,s  )||_{\mathbb{X}} \leq \lambda(t,s)\, \, ,  \, \, (t,s)\in \mathbb{R}\times \mathbb{R}$. This Lemma will be invoked in section \ref{section4}.  Further details on almost automorphic functions can be found in \cite{diaganaBook2013almost,16} and on exponential dichotomy in \cite{06,hartmanBook2002ordinary}.
\section{Bi-compact almost automorphy of the Green function $G(\cdot , \cdot)$}\label{section3}
Our main result in this section is the following Theorem:
\begin{thm}\label{teo001}
 Suppose that system (\ref{eq3}) has an $(\alpha, K, P)$-exponential  dichotomy with fundamental matrix $\Phi(t)$,
$\Phi(0)=I$ and $A(\cdot)$ is compact almost automorphic, that is:
Given and arbitrary sequence $\set{s'_n}\subset \R$ there exists a
subsequence $\set{s_n}\subset \set{s'_n}$ such that the following
limits hold uniformly on compacts sets of the real line
\begin{equation}\label{eq4}
\lim_{n\to +\infty} A(t+s_n) = :B(t),\quad \lim_{n\to
+\infty}B(t-s_n)= A(t)\, .
\end{equation}
Then, there exists a projection matrix $P_0$, such that the system
\begin{equation}\label{eq5}
y'(t)=B(t)y(t)\, ,
\end{equation}
has an  $(\alpha, K, P_0)$-exponential  dichotomy. Furthermore, the associated Green function of system (\ref{eq3}) is compact Bi-almost automorphic.
\end{thm}

Before to prove Theorem \ref{teo001}, we stablish the following Lemma:

\begin{lem}\label{lemad01} Under hypothesis of Theorem \ref{teo001}, there exists a subsequence $\{\zeta_n\} \subset \{s_n\}$ such that, the following limits hold
\begin{equation*}
\lim_{n\to + \infty}\sup_{t\in \mathcal{C}}||\Phi_n(t)-\Psi(t)||=0,\quad \lim_{n\to +
\infty}\sup_{t\in \mathcal{C}}||\Phi_n^{-1}(t)-\Psi^{-1}(t)||=0\, ,
\end{equation*}
%
where $\mathcal{C}\subset \mathbb{R}$ is compact, $\Phi_n(t)=\Phi(t+\zeta_n)\Phi^{-1}(\zeta_n)$ and $\Psi $ is a fundamental matrix of (\ref{eq5}).
\end{lem}

\begin{proof}
(\textbf{Proof of Lemma \ref{lemad01}})

It is sufficient to consider the compact interval $\mathcal{C}=[-a, a]$, with $a>0$.

From hypothesis, there exists a subsequence $\{s_n\} \subset \{s_n'\}$ such that the limits in (\ref{eq4}) hold. For this subsequence, let us consider the sequence $\Phi_n^{*}(t)=\Phi(t+s_n)\Phi^{-1}(s_n)$, $n\in \mathbb{N}$; and notice that for each $n \in\N$, $\Phi_n^*(0)=I$ and $ \Phi_n^{*}(\cdot)$ is a fundamental matrix of the system:
\begin{equation*}
x'(t)=A(t+s_n)x(t)\, .
\end{equation*}
Then,  by direct integration it follows that
\begin{equation*}
\Phi_n^*(t)= I+\int_0^tA(u+s_n)\Phi_n^*(u)du.
\end{equation*}
Let $M>0$ be the supremum of $A(\cdot)$; then, by the Gronwall-Bellman Lemma we have
$$\norm{\Phi_n^*(t)}\leq \norm{I}e^{Mt}\leq \norm{I}e^{Ma}\, .$$
That is, $\Phi_n^*$ is uniformly bounded on $\mathcal{C}$. Now, since
\begin{equation}\label{5.3.6}
\left( \Phi_n^*(t)\right)' = A(t+s_n)\Phi_n^*(t),\quad \text{for $n\in\N$  and $t\in \mathcal{C}$}\, ,
 \end{equation}
then the sequence of derivatives $\left( \Phi_n^*(\cdot)\right)'$  is uniformly bounded over $\mathcal{C}$.
It follows from Arzela-Ascoli Theorem that
$\set{\Phi_n^*}$ has a subsequence $\set{\Phi_n}$ (i.e. there
exists an associated subsequence  $\set{\zeta_n} \subseteq
\set{s_n}$) such that $\Phi_n(t)=\Phi(t+\zeta_n)\Phi^{-1}(\zeta_n)$ and is uniformly convergent on $\mathcal{C}$.
That is, there exists $\Psi$ such that
\begin{equation}\label{3.5.37}
 \lim_{n\to +\infty} \Phi_n=\Psi,\;\text{uniformly on $\mathcal{C}$.}
 \end{equation}
On the other hand, since the subsequence $\{\Phi_n\}$ satisfies (\ref{5.3.6}), then
\begin{equation*}
\Phi'_n(t)=A(t+s_n) \Phi_n(t),\quad n\in\N,\; t\in \mathcal{C} .
\end{equation*}
Therefore, $ \Phi'_n$ converges uniformly over $\mathcal{C}$ to $B(t)\Psi(t)$. Consequently, $\Psi(\cdot)$ is differentiable and moreover
\begin{equation*}
\lim_{n\to +\infty}\Phi'_n(t)=\Psi'(t)\;\text{over $\mathcal{C}$ .}
\end{equation*}
Since $\Psi'(t)=B(t) \Psi (t)$ and $\Psi(0)=I$, then $\Psi$  is a fundamental matrix of (\ref{eq5}). That is, $\Psi (\cdot)$ is non singular, and hence  $\Psi^{-1}$ exists  on  $\mathcal{C}$.\\
Now we prove that the following limit holds uniformly on $\mathcal{C}$
$$ \lim_{n\to +\infty}\Phi^{-1}_n(t)=\Psi^{-1}(t)\, .$$
Firstly, notice that $\Phi^{-1}_n$
satisfies the equation
$$ x'(t)= -x(t)A(t+\zeta_n)\, ;$$
and, by direct integration we have
$$\Phi^{-1}_n(t)= I - \int_0^t\Phi^{-1}_n(u)A(u+\zeta_n)du .$$
Similarly,
$$\Psi^{-1}(t)= I - \int_0^t\Psi^{-1}(u)B(u)du\, .$$
Then,
\begin{eqnarray}
\norm{\Phi_n^{-1}(t)-\Psi^{-1}(t)}&\leq & \int_0^t \norm{ \Psi^{-1}(u)B(u)
-\Phi^{-1}_n(u)A(u+\zeta_n)}du \nonumber \\
&\leq& \int_0^t \norm{\Psi^{-1}(u) -\Phi^{-1}_n(u)}du\norm{A}_\infty + \nonumber \\
& + & \int_0^t\norm{A(u+\zeta_n)-B(u)}du\norm{\Psi^{-1}}_\infty . \label{EqqMatrix}
\end{eqnarray}
Now, since the limits in (\ref{eq4}) are uniform in $\mathcal{C}$; given $\epsilon >0$, there exists $N_0\in\N$ such that for all $n\geq N_0$ :
$$ \sup_{t\in \mathcal{C}}\norm{A(t+\zeta_n)-B(t)}<\epsilon\,.$$
Therefore, from inequality (\ref{EqqMatrix}) we have, for $n\geq N_0$
\begin{align*}
\norm{\Phi_n^{-1}(t)-\Psi^{-1}(t)}\leq \epsilon
a\norm{\Psi^{-1}}_\infty+ \int_0^t \norm{\Psi^{-1}(u)
-\Phi^{-1}_n(u)}du\,\norm{A}_\infty\, ;
\end{align*}
and, from the Gronwall-Bellman Lemma again, we conclude
\begin{align*}
\norm{\Phi_n^{-1}(t)-\Psi^{-1}(t)}\leq \epsilon
a\norm{\Psi^{-1}}_\infty e^{\norm{A}_\infty a}\, .
\end{align*}
\end{proof}
Next, we prove the main result of this section.
\begin{proof}
(\textbf{Proof of Theorem \ref{teo001}})

\noindent Let  $\Phi$ be a fundamental matrix of
(\ref{eq3}) and
\begin{equation}\label{eq31}
G(t,s):=
\begin{cases}
\Phi(t)P\Phi^{-1}(s),\quad &\text{$t\geq s$ ,}\\
-\Phi(t)(I-P)\Phi^{-1}(s),\quad &\text{$t < s$} .
\end{cases}
\end{equation}
its associated Green function such that $||G(t,s)||\leq Ke^{-\alpha\abs{t-s}}$,\; $t,s\in\R$.\\
\noindent The proof will be divided into three steps.

\noindent \textbf{Step 1. The system (\ref{eq5}) has an $(\alpha, K, P_0)$-exponential dichotomy}.

Because of Lemma \ref{lemad01}, there exists a subsequence $\{\zeta_n\} \subset \{s_n\}$ such that the following limits are uniform on compact subsets of $\mathbb{R}$:
\begin{equation}
\lim_{n\to +\infty}\Phi_n(t)=\Psi(t)\, \, \, \, {\rm and}\, \, \, \, \lim_{n\to +\infty}\Phi_n^{-1}(t)=\Psi^{-1}(t)\, ,
\end{equation}
where $\Phi_n(t)=\Phi(t+\zeta_n)\Phi^{-1}(\zeta_n)$ and $\Psi $ is a fundamental matrix of (\ref{eq5}).

%

Note that $\norm{\Phi(\zeta_n)P\Phi^{-1}(\zeta_n)} \leq
K$\;and\;$\norm{ \Phi(\zeta_n)Q\Phi^{-1}(\zeta_n)} \leq K$, where $Q=I-P$. Thus, the sequences of matrices $\{\Phi(\zeta_n)P\Phi^{-1}(\zeta_n) \} $  and $ \{ \Phi(\zeta_n)Q\Phi^{-1}(\zeta_n) \}$ are bounded; therefore,
there exists a subsequence
$\set{\eta_n}\subset \set{\zeta_n}$ such that:
\begin{align*}
&\lim_{n\to +\infty}\Phi(\eta_n)P\Phi^{-1}(\eta_n)=: P_0\, ,\\
&\lim_{n\to +\infty}\Phi(\eta_n)Q\Phi^{-1}(\eta_n)=: Q_0\, .
\end{align*}
Notice that\, $P_0^2=P_0$ and $P_0+Q_0=I$. Now, taking $n\to +\infty$  in the following inequalities:
\begin{align*}
&\norm{\left( \Phi(t+\eta_n)\Phi^{-1}(\eta_n) \right) \Phi(\eta_n)P\Phi^{-1}(\eta_n) \left( \Phi(s+\eta_n)\Phi^{-1}(\eta_n) \right)^{-1}}\leq
Ke^{-\alpha(t-s)},\quad t\geq s\\
&\norm{\left( \Phi(t+\eta_n)\Phi^{-1}(\eta_n) \right) \Phi(\eta_n)Q\Phi^{-1}(\eta_n) \left( \Phi(s+\eta_n)\Phi^{-1}(\eta_n) \right)^{-1}}\leq
Ke^{-\alpha(s-t)},\quad t < s\, ,
\end{align*}
we conclude that
\begin{align*}
&\norm{\Psi(t)P_0\Psi^{-1}(s)}\leq Ke^{-\alpha(t-s)},\; t\geq s\, ,\\
&\norm{\Psi(t)Q_0\Psi^{-1}(s)}\leq Ke^{-\alpha(s-t)},\; t < s\, .
\end{align*}
That is, there exists a projection matrix $P_0$, such that the homogeneous system (\ref{eq5}) has an $(\alpha, K, P_0)$-exponential dichotomy with Green function:
\begin{equation}\label{eq32}
\tilde{G}(t,s):=
\begin{cases}
\Psi(t)P_0\Psi^{-1}(s),\quad &\text{$t\geq s$\, ,}\\
-\Psi(t)Q_0\Psi^{-1}(s),\quad &\text{$t < s$}\, .
\end{cases}
\end{equation}

\noindent \textbf{Step 2. The Green function $G(\cdot, \cdot)$   defined in (\ref{eq31}) is Bi-almost automorphic}.

Let $\{ s_n'\}$ be an arbitrary sequence. Then, from hypothesis and \textbf{Step 1}, there exists a subsequence $\{\eta_n\} \subset \{s_n'\}$ such that the following pointwise limit holds
$$\lim_{n\to +\infty}G(t+\eta_n, s+\eta_n)=\tilde{G}(t,s)\, ;$$
where,
\begin{equation}\label{eq33}
{G}(t+\eta_n,s+\eta_n):=
\begin{cases}
\Phi_n(t)P_n\Phi_n^{-1}(s),\quad &\text{$t\geq s$}\\
-\Phi_n(t)Q_n\Phi_n^{-1}(s),\quad &\text{$t < s$}\, ,
\end{cases}
\end{equation}
$\Phi_n(t)=\Phi(t+\eta_n)\Phi^{-1}(\eta_n)$, $P_n:=\Phi(\eta_n)P\Phi^{-1}(\eta_n),\,\, Q_n:=\Phi(\eta_n)Q\Phi^{-1}(\eta_n)$
 and $\tilde{G}(\cdot, \cdot)$ was defined in (\ref{eq32}).


We claim that, there exists a subsequence $\set{\xi_n}\subseteq \set{\eta_n}$ such
that the following limit is attained
$$\lim_{n\to +\infty}\tilde{G}(t-\xi_n, s-\xi_n)=G(t,s)\, .$$
In fact, since $\Psi( \cdot )$, with $\Psi(0)=I$,  is a fundamental matrix of (\ref{eq5}), then for each $n\in \mathbb{N}$, $\Psi_n(t):=\Psi(t-\eta_n)\Psi^{-1}(-\eta_n)$  is a fundamental matrix of the system
$$ z'(t)=B(t-\eta_n)z\, .$$
Arguing as in the \textbf{Step 1}, we can find a subsequence $\{\xi_n\} \subset \{ \eta_n \}$ such that
%
%
%
\begin{align*}
&\lim_{n\to +\infty}\Psi(-\xi_n)P_0\Psi^{-1}(-\xi_n)=:\tilde{P}_1,\\
&\lim_{n\to +\infty}\Psi(-\xi_n)Q_0\Psi^{-1}(-\xi_n)=: \tilde{Q}_1\, .
\end{align*}
Notice that $\tilde{P}_1^2= \tilde{P}_1$,\, $\tilde{P}_1+
\tilde{Q}_1=I$; and, as in the proof of Lemma \ref{lemad01}, we have that the limits $\lim_{n\to +\infty}B(t-\xi_n)= A(t)$ and  $\lim_{n\to
+\infty}\Psi_n(t)= \Upsilon(t)$ are uniform on compact subsets of $\R$, where \,$\Upsilon(t)$\, is a
fundamental matrix of system (\ref{eq3}), note that $\Upsilon(0)=I$. Then,
there exists a non singular matrix $C$ such that \, $\Upsilon(t)=\Phi(t)C$,\,
and because $\Upsilon(0)=\Phi(0)=I$, we have   $C=I$,  therefore \,
$\Upsilon(t)=\Phi(t)$. On the other hand, we know that the projection for an exponential dichotomy is unique, then \,$\tilde{P}_1=P$ and
$\tilde{Q}_1=Q$. All this reasoning implies that
$$ \lim_{n\to + \infty}\tilde{G}(t-\xi_n,s-\xi_n)=G(t,s)\, , $$
for each point $(t,s)\in \mathbb{R}^2$, as we claimed.
%

\noindent \textbf{Step 3: The Green function $G(\cdot, \cdot)$   defined in (\ref{eq31}) is compact Bi-almost automorphic}.

From  Lemma \ref{lemad01}, there exists a subsequence $\{\eta_n\} \subset \{s_n\}$ such that if $\mathcal{C}=[a,b]$ is a compact subset  of $\mathbb{R}$, then the following limits hold:
\begin{equation}\label{NewConvComp}
\lim_{n\to +\infty}\sup_{t\in \mathcal{C}}\norm{\Phi_n(t) -\Psi(t)}=0\, \, \, \, {\rm and}\, \, \, \, \lim_{n\to +\infty}\sup_{t\in \mathcal{C}}\norm{\Phi_n^{-1}(t)-\Psi^{-1}(t)}=0\, ,
\end{equation}
where, $\Phi_n(t)=\Phi(t+\eta_n)\Phi^{-1}(\eta_n)$ and $\Psi $ is a fundamental matrix of the system in (\ref{eq5}). Also, there exist positive constants $C_1,$ $C_2,$ $C_3$, such that for every $n\in \mathbb{N}$ and  $t\in \mathcal{C}$, we have:
\begin{equation}\label{EeQqNn}
\norm{\Phi_n(t)}\le C_1,\; \norm{\Phi_n^{-1}(t)}\le C_2, \; \norm{\Psi(t)}\le C_3\, .
\end{equation}
Furthermore,
\begin{align}
&P_n=\Phi(\eta_n)P\Phi^{-1}(\eta_n)\to P_0\, ,\, \,  \text{as  $n\to +\infty$\, ,} \label{CcOoNn01}
\\
&Q_n=\Phi(\eta_n)Q\Phi^{-1}(\eta_n)\to Q_0\, ,\, \,  \text{as  $n\to +\infty$\, .} \label{CcOoNn02}
\end{align}
We claim that, if $\mathbb{K}=[a,b]\times [a,b]\subset \mathbb{R}^2$, then
\begin{equation}\label{EeeQqqNnn}
\lim_{n\to +\infty}\sup_{(t,s)\in \mathbb{K}}||G(t+\eta_n,s+\eta_n)-\tilde{G}(t,s)||=0\, ,
\end{equation}
where, $G(t+\eta_n,s+\eta_n)$ was defined in (\ref{eq33}). In fact,

\noindent 1) If $t\geq s$, then
%
$$\norm{G(t+s_n,s+s_n)-\tilde{G}(t,s)}=\norm{\Phi_n(t)P_n\Phi_n^{-1}(s)-\Psi(t)P_0\Psi^{-1}(s)}=$$
    $$=\norm{\Phi_n(t)(P_n-P_0+P_0)\Phi_n^{-1}(s)-\Psi(t)P_0\Psi^{-1}(s)}=$$
    $$=\norm{\Phi_n(t)(P_n-P_0)\Phi_n^{-1}(s)+\Phi_n(t)P_0\Phi_n^{-1}(s)-\Psi(t)P_0\Psi^{-1}(s)}=$$
    $$=\norm{\Phi_n(t)(P_n-P_0)\Phi_n^{-1}(s)+(\Phi_n(t)-\Psi(t)+\Psi(t))P_0\Phi_n^{-1}(s)-\Psi(t)P_0\Psi^{-1}(s)}=$$
    $$=\norm{\Phi_n(t)(P_n-P_0)\Phi_n^{-1}(s)+(\Phi_n(t)-\Psi(t))P_0\Phi_n^{-1}(s)+\Psi(t)P_0\Phi_n^{-1}(s)-\Psi(t)P_0\Psi^{-1}(s)}=$$
    $$=\norm{\Phi_n(t)(P_n-P_0)\Phi_n^{-1}(s)+(\Phi_n(t)-\Psi(t))P_0\Phi_n^{-1}(s)+\Psi(t)P_0(\Phi_n^{-1}(s)-\Psi^{-1}(s))}\le$$
    $$\le \norm{\Phi_n(t)}\, \norm{P_n-P_0}\, \norm{\Phi_n^{-1}(s)}+\norm{\Phi_n(t)-\Psi(t)}\,\norm{P_0}\,\norm{\Phi_n^{-1}(s)}+$$
    $$+\norm{\Psi(t)}\,\norm{P_0}\,\norm{\Phi_n^{-1}(s)-\Psi^{-1}(s)}<$$
    $$<A'\norm{P_n-P_0}+B'\norm{\Phi_n(t)-\Psi(t)}+C'\norm{\Phi_n^{-1}(s)-\Psi^{-1}(s)}\, ,$$
for some positive constants $A^{'},B^{'}, C^{'}$. Thus,
\begin{eqnarray*}
\norm{G(t+s_n,s+s_n)-\tilde{G}(t,s)}& < & A'\norm{P_n-P_0}+B'\sup_{t\in [a,b]}\norm{\Phi_n(t)-\Psi(t)}+\\
&+&  C'\sup_{s\in [a,b]}\norm{\Phi_n^{-1}(s)-\Psi^{-1}(s)}\, .
\end{eqnarray*}

\noindent 2) If $t<s$, then (as in the previous computation) we have
\begin{eqnarray*}
\norm{G(t+s_n,s+s_n)-\tilde{G}(t,s)}& = &\norm{\Phi_n(t)Q_n\Phi_n^{-1}(s)-\Psi(t)Q_0\Psi^{-1}(s)}<\\
&<&A^{'}\norm{Q_n-Q_0} +  B^{'}\sup_{t\in [a,b]} \norm{\Phi_n(t)-\Psi(t)}+ \\
&+& C^{'}\sup_{s\in [a,b]}\norm{\Phi_n^{-1}(s)-\Psi^{-1}(s)}\, ,
\end{eqnarray*}
Therefore, from 1), 2), (\ref{NewConvComp}), (\ref{CcOoNn01}) and (\ref{CcOoNn02}) we conclude (\ref{EeeQqqNnn}).

Analogously, a subsequence $\{\xi_n\}\subset\{\eta_n\}$ can be obtained  such that, on compact sets  $\mathbb{K}\subset \mathbb{R}^2$, the following limit holds
$$\underset{n\to +\infty}{\lim}\sup_{(t,s)\in \mathbb{K}}||\tilde{G}(t-\xi_n,s-\xi_n)-G(t,s)||=0\, .$$

\end{proof}

\begin{rem} In the work \cite{coronel2016almost}, the authors assume that the fundamental matrix solution $
\Phi$ of the system 
$$x'(t)=A(t)x(t)\, ,$$
commutes with the projection $P$ in order to prove the integral Bi-almost automorphy of the Green function $G(\cdot, \cdot)$; in our result, namely Theorem \ref{teo001}, this hypothesis is not necessary.
\end{rem}

\begin{rem} From {\bf Step 1} and {\bf Step 2} of the previous proof, we conclude that Theorem \ref{teo001} can be stated in the almost automorphic category; thus, for systems
$$x'(t)=A(t)x(t)\, ,$$
in which $A(\cdot)$ is an almost automorphic matrix.
\end{rem}

\begin{rem}
Although $G(\cdot, \cdot)$ is measurable,  discontinuous  and compact Bi-almost automorphic, it does not contradicts the result of Veech \cite{veech1969complementation}, where he proved that \textbf{on locally compact groups, Haar measurable almost automorphic functions are continuous}. To the contrary, this put in light a crucial difference between fully almost automorhic functions, where we take sequences in the full euclidean group $\mathbb{R}^2$ (in this case), and the measurable (compact) Bi-almost automorphic ones, in which the sequences belongs to the proper subgroup $\Delta_2 \subset \mathbb{R}^2$, which is the principal diagonal of $\mathbb{R}^2$.
\end{rem}

The following Lemma is crucial in the proof of invariance of the compact almost automorphic function space under convolution products whose kernel is the Green function $G(\cdot, \cdot)$ (Theorem \ref{InvConv}).

\begin{lem}\label{almostUC} (\textbf{$\Delta_2$-like uniform continuity of the Green function $G(\cdot , \cdot)$}) Let $\{t_n\}$ and $ \{s_n\}$ be real sequences such that $|t_n-s_n|\to 0,$ as $n\to +\infty$; then, for each $(t,s) \in \mathbb{R}^2$ we have  $\norm{G(t+t_n,s+t_n)-G(t+s_n,s+s_n)} \to 0,$ as $n\to +\infty$.
\end{lem}
\begin{proof}
Let us denote by $\mathcal{U}:=\{(t,s) \in \mathbb{R}^2\, :\, t < s\}$ and by $\mathcal{V}:=\{(t,s) \in \mathbb{R}^2\, :\, t \geq s\}$, also define $G^2(t,s):=-\Phi(t)Q\Phi^{-1}(s)$ and $G^1(t,s):=\Phi(t)P\Phi^{-1}(s)$. Since $\mathcal{U}, \mathcal{V}$ constitute a partition of $\mathbb{R}^2$, then any point $(t,s)$ is either in $\mathcal{U}$ or in $\mathcal{V}$. Let us suppose that $(t,s)\in \mathcal{U}$ and take the sequence
\begin{equation}
\alpha_n(t,s)=\norm{G^2(t+t_n,s+t_n)-G^2(t+s_n,s+s_n)}\, .
\end{equation}
We claim that $\alpha_{n}(t,s) \to 0$, when $n\to +\infty$. In fact, let $\{\alpha_n'(t,s)\} \subset \{\alpha_n(t,s)\}$ be a convergent subsequence which converges to $\psi_0(t,s)$, where
$$ \alpha_n'(t,s)=\norm{G^2(t+t_n',s+t_n')-G^2(t+s_n',s+s_n')}\, , $$
where $ \{t_n'\} \subset \{t_n\}$ and $\{s_n'\} \subset \{s_n\}$. We must prove that $\psi_0(t,s)=0$.

Since $G^2(\cdot , \cdot )$ is compact Bi-almost automorphic on $\mathcal{U} $, there exists a subsequence $\{s_n''\} \subset \{s_n'\}$ and a function $\tilde{G}^2$ such that the following limits holds:
\begin{eqnarray}
\lim_{n\to +\infty} \sup_{(t,s) \in \mathbb{K}} \norm{ G^2(t+s_n'',s+s_n'')- \tilde{G}^2(t,s) } &=&0\, , \label{EQEx01}\\
\lim_{n\to +\infty} \sup_{(t,s) \in \mathbb{K}} \norm{ \tilde{G}^2(t-s_n'',s-s_n'')- G^2(t,s) }&=&0\nonumber \, ;
\end{eqnarray}
where, $\mathbb{K}$ is a compact subset of $\mathcal{U}$. As a byproduct, the limit function $\tilde{G}^2$ is continuous on $\mathcal{U}$.

Since $\{|t_n-s_n|\}$ is bounded, there exists $R>0$ such that $t_n-s_n \in [-R,R]$ (a compact interval in $\mathbb{R}$). Therefore the points $(t_n-s_n,t_n-s_n)\in L_R$, where $L_R:=\{(z,z)\in \mathbb{R}^2\, :\, |z|\leq \sqrt{2}R\}$ a compact subset of (the diagonal of) $\mathbb{R}^2$, which implies that $L_R(t,s):=(t,s)+L_R$ is a compact subset of $\mathcal{U}$. Now, from the inequalities:
\begin{eqnarray*}
\alpha_{n}''(t,s) & =  & \norm{ G^2(t+t_n'',s+t_n'')-G^2(t+s_n'',s+s_n'')} \\
& \leq & \norm{ G^2(t+t_n''-s_n''+s_n'',s+t_n''-s_n''+s_n'')-\tilde{G}^2(t+t_n''- s_n'',s+t_n''- s_n'') } + \\
&+& \norm{ \tilde{G}^2(t+t_n''- s_n'',s+t_n''- s_n'')-\tilde{G}^2(t,s)}+\norm{ \tilde{G}^2(t,s)-G^2(t+s_n'',s+s_n'') }\\
&\leq & \sup_{(z_1,z_2)\in L_R(t,s)}\norm{ G^2(z_1+s_n'',z_2+s_n'')-\tilde{G}^2(z_1,z_2) }+\\
&+& \norm{ \tilde{G}^2(t+t_n''- s_n'',s+t_n''- s_n'')-\tilde{G}^2(t,s)}+\norm{\tilde{G}^2(t,s)-G^2(t+s_n'',s+s_n'') }\, ,
\end{eqnarray*}
from (\ref{EQEx01}) and the continuity of $\tilde{G}^2$ we conclude that $\alpha_{n}''(t,s) \to 0$, as $n\to +\infty$. This implies that $\psi_0(t,s)=0$; therefore, $\alpha_{n}(t,s) \to 0$ when $n\to +\infty$, as claimed.

Analogously, it can be shown that if $(t,s)\in \mathcal{V}$, then
\begin{equation*}
\lim_{n\to +\infty} \norm{ G^1(t+t_n,s+t_n)-G^1(t+s_n,s+s_n) }=0\, .
\end{equation*}
\end{proof}
In the proof of this Lemma, only the fact that each branch of the Green function, namely the functions  $G^1$ and $G^2$, being compact Bi-almost automorphic in their respective domains, was utilized, without requiring the full compact Bi-almost automorphy of $G(\cdot , \cdot)$.
That is, the previous Lemma is true only assuming that $G(\cdot , \cdot)$ is \textit{compact Bi-almost automorphic relative to the following collection}
$$ \Lambda:=\Big{ \{ }  \mathbb{K} \subset \mathbb{R}^2\, : \, \mathbb{K}\, \textit{is   compact  and } \mathbb{K} \cap \Delta_2 = \emptyset \, \Big{ \} } .$$
 \textit{Compact Bi-almost automorphic relative to} $\Lambda$, means that in Definition \ref{CompBiAAMeasu}  compact sets are taken in $\Lambda$.
%

\section{Compact almost automorphic solutions of non-autonomous differential equations}\label{section4}

\subsection{invariance of $\mathcal{K}AA$ under convolution product with kernel $G(\cdot , \cdot)$}

\begin{thm}\label{InvConv}
Let $A(\cdot)$ be a compact almost automorphic matrix and suppose that system (\ref{eq3}) has an $(\alpha, K, P)$-exponential dichotomy with Green function $G(\cdot,\cdot)$. Then the operator $\mathcal{G}_1$ defined by 
\begin{equation}\label{Operator1}
\mathcal{G}_1 u (t):=\int_{\R}G(t,s)u(s)ds\, ,
\end{equation}
leaves invariant the space $\mathcal{K}AA(\mathbb{R}; \mathbb{C}^p)$.\\
%
%
\end{thm}
\begin{proof}
Let  $u\in \mathcal{K}AA(\mathbb{R}; \mathbb{C}^p)$, by Proposition \ref{CharCAA} it is sufficient to prove that $\mathcal{G} u$ is almost automorphic and uniformly continuous. We will accomplish this in two steps.

\noindent \textbf{Step 1: $\mathcal{G} u$ is almost automorphic}. Let $\{s'_n\}$ be any real sequence. Since $G(\cdot,\cdot)$ is  compact  Bi-almost automorphic and $u$ is compact almost automorphic, there exists a subsequence  $\{s_n\}\subseteq \{s'_n\}$ such that on compact subsets of
$\mathbb{R}\times \mathbb{R}$ the following uniform limits holds:
$$\lim\limits_{n\to +\infty} G(t+s_n,s+s_n)=:\tilde{G}(t,s),\quad \lim\limits_{n\to +\infty} \tilde{G}(t-s_n,s-s_n)=G(t,s);$$
and on compact subsets of $\mathbb{R}$ the following uniform limits holds
$$\lim\limits_{n\to +\infty} u(t+s_n)=:\tilde{u}(t),\quad \lim\limits_{n\to +\infty} \tilde{u}(t-s_n)=u(t).$$

In particular the previous limits are pointwise. Let $v(t)=\mathcal{G} u(t)$; then, since $G(\cdot,\cdot)$ is $\lambda$-bounded, for $\lambda(t,s)=Ke^{-\alpha|t-s|}\, , t,s \in \mathbb{R}$, by Lemma \ref{lemBou} and using the Lebesgue's Dominate Convergence Theorem, we have
\begin{equation}\label{Eq001}
\lim\limits_{n\to +\infty} v(t+s_n)=\tilde{v}(t),\, \, \lim_{n\to +\infty}\tilde{v}(t-s_n)=v(t)\, ,
\end{equation}
where
$$\tilde{v}(t):=\int_{-\infty}^{+\infty} \tilde{G}(t,s)\tilde{u}(s)ds.$$
What we have proved is that $v=\mathcal{G} u$ is almost automorphic.

\noindent \textbf{Step 2: $\mathcal{G} u$ is uniformly continuous}. Let $\{t_n\},\{s_n \}$ be two sequences in $\mathbb{R}$ such that $|t_n-s_n| \to 0$ when $n\to +\infty$, then
\begin{eqnarray*}
|\mathcal{G} u(t_n)-\mathcal{G} u(s_n)|&\leq & \int_{\mathbb{R}}\norm{G(t_n, s+t_n)-G(s_n,s+s_n)}|u(s+t_n)|ds \\
&+& \int_{\mathbb{R}}\norm{G(s_n,s+s_n)} |u(s+t_n)-u(s+s_n)|ds\\
&:=&I_1(n)+I_2(n)\, .
\end{eqnarray*}
Now, by Lemma \ref{lemBou}, the $\Delta_2$-like uniform continuity of $G(\cdot,\cdot)$, the uniform continuity of $u$ and the Lebesgue's Dominated Convergence Theorem, we conclude
$$\lim_{n\to +\infty}I_1(n)=0 \, =\, \lim_{n\to +\infty}I_2(n) \, .$$
\end{proof}
\noindent The proof of the following Theorem, is analogous to the one of Theorem \ref{InvConv}.

\begin{thm}\label{InvConv1}
Let $A(\cdot)$ be a compact almost automorphic matrix and suppose that system (\ref{eq3}) has an $(\alpha, K, P)$-exponential dichotomy with Green function $G(\cdot,\cdot)$. Then the operator $\mathcal{G}_2$ defined by 
\begin{equation}\label{Operator2}
\mathcal{G}_2 u (t):=\int_{-\infty}^t G(t,s)u(s)ds\, ,
\end{equation}
leaves invariant the space $\mathcal{K}AA(\mathbb{R}; \mathbb{C}^p)$.\\
\end{thm}
\begin{rem} In contrast with the almost automorphic scenario, it is not enough that $G(\cdot,\cdot)$ be Bi-almost automorphy and $\lambda$-bounded, for $\lambda(t,s)=Ke^{-\alpha|t-s|}$ (for instance), to show the invariance of the space $\mathcal{K}AA(\mathbb{R}; \mathbb{C}^p)$ under the convolution products in Theorems \ref{InvConv} and \ref{InvConv1}; it seems very important also the property of being \textbf{$\Delta_2$-like uniformly continuous}, which is a property that meets the Green function $G(\cdot,\cdot)$. The \textbf{$\Delta_2$-like uniform  continuity} of $G(\cdot,\cdot)$ is a hypothesis that has not been considered in \cite[Corillary 1]{abbas2021pseudo}. Note that, the aforementioned hypothesis cannot be concluded from the Bi-almost automorphy and $\lambda$-boundedness of $G(\cdot,\cdot)$. 

\end{rem}

\noindent The proof of the following Proposition is immediate
\begin{prop}\label{PropCBiaa}
If $a(\cdot)$ is compact almost automorphic, then 
$$G_1(t,s):=\int_{s}^t a(\xi)d\xi\, , \, \ {\rm and}\, \,  G_2(t,s):=\exp \left( \int_{s}^t a(\xi)d\xi  \right)\, ,$$
are Bi-almost automorphics and $\Delta_2$-like uniformly continuous.
\end{prop}


\subsection{compact almost automorphic solutions}
Finally, the following theorem arrives
\begin{thm}
If the linear system (\ref{eq3}) has an $(\alpha, K, P)$-exponential
dichotomy with Green function $G(\cdot , \cdot)$, then the system 
$$x'(t)=A(t)x(t)+f(t)\, $$
where $A(\cdot)$ is a compact almost automorphic matrix and $f \in \mathcal{K}AA(\mathbb{R}; \mathbb{C}^p)$, has a unique compact almost automorphic solution $x$, given by
$$x(t)=\int_{\R}G(t,s)f(s)ds\, .$$
Moreover, if $\Phi(\cdot)$ is the fundamental solution of system (\ref{eq3}) and $P$ the projection, then, the solution $x$ is given by
$$x(t)=\int_{-\infty}^t \Phi(t)P\Phi^{-1}(s)f(s)ds+\int_{t}^{\infty} \Phi(t)(I-P)\Phi^{-1}(s)f(s)ds\, .$$
\end{thm}
\noindent Obviously, the previous Proposition can be stated for almost automorphic systems, instead of purely compact almost automorphic ones.




\section{The existence of a positive compact almost automorphic solutions of delay differential equations of biological interest.}\label{section5}
In this section, in the spirit of the papers \cite{abbas2021pseudo,ben2020positive}, we study the existence of a  positive compact almost automorphic solution of the following biological model with non-linear harvesting terms and mixed delays 
\begin{equation}\label{eq11-ps}
\dot{u}(t)=
-\alpha(t)u(t)+\sum_{i=1}^n\beta_i(t)f_i(\lambda_i(t)u(t-\tau_i(t))+b(t)H(u(t-\sigma(t)))\, .
\end{equation}
This model codify, at least, the cases of:
\begin{enumerate}
\item Nicholson: $f_i(z)=ze^{-z}\, .$
\item Lasota-Wazewska: $f_i(z)=e^{-z}\, .$
\item Mackey-Glass: $f_{i,m}(z)=\dfrac{z}{1+z^m}\, .$
\end{enumerate}
Note that, in equation (\ref{eq11-ps}), we are considering a harvesting term with variable delay, this case was not considered in the work \cite{abbas2021pseudo}, but it has been considered in the work \cite{ben2020positive}.

The following are the necessary hypothesis on the coefficients of
(\ref{eq11-ps}) to obtain our results:
\begin{itemize}
\item[\bf A1] The functions $\alpha(t)$, $b(t),\sigma(t)$, $\beta_i(t)$,
$\lambda_i(t)$, $\tau_i(t)$ for $i =1, 2, \cdots, n$ are positive compact almost automorphic functions on $\mathbb{R}$.
\item[\bf A2] The functions $\alpha(\cdot)$, $\beta_i(\cdot)$, $\lambda_i(\cdot)$ are bounded away from zero.
\item[\bf A3] The function $H:\R^+_0 \mapsto \R^+_0$ is bounded and  Lipschitz continuous, i.e. there exists a positive constant $L_H$
such that:
$$ |H(u)-H(v)| < L_H|u-v|\quad\text{for $u,v \in\R^+_0$}$$
\end{itemize}
If $f$ is a real bounded continuous function, defined in its domain $D_f\subset \mathbb{R}$, then $\overline{f}$ and $\underline{f} $ are defined by
$$\overline{f}:=\sup_{t\in D_f}\{ f(t)\} \, , \, \, \underline{f}:=\inf_{t\in D_f}\{ f(t)\}\, .$$
Also define the positive constants $\tau$ and $\overline{\lambda}$ as follows 
$$\tau= \max \{ \max_{1\leq i \leq n} \overline{ \tau_i}, \overline{\sigma}\}\, ,\, \, \displaystyle\overline{\lambda}= \max_{1\leq 1\leq
n}\set{\overline{\lambda_i}}\, .$$ 
Denote by $C([-\tau,0], \mathbb{R})$ the Banach space of continuous functions from $[-\tau,0]$ to $ \mathbb{R}$ under the norm $||\phi||_{\tau}=\sup_{\theta \in [-\tau,0]}|\phi(\theta)|$. Let $\mathcal{C}^+$ be the cone of non-negative functions in $C([-\tau,0], \mathbb{R})$, that is
$$\mathcal{C}^+=\Big{\{ } \phi \in  C([-\tau,0], \mathbb{R})\, : \,  \phi(t)\geq 0\Big{ \} }\, ,$$
also define the set
$$\mathcal{C}_0^+=\Big{\{ } \phi \in \mathcal{C}^+ \, : \,  \phi(0)> 0 \Big{ \} }\, .$$
Due to biological interpretations, the set of admissible
initial conditions for equation (\ref{eq11-ps}) is considered in $\mathcal{C}_0^+$.

If $x(\cdot)$ is defined in the interval $[t_0-\tau,\sigma]$ with $t_0, \sigma \in \mathbb{R}$, then the function $x_t \in C([-\tau,0], \mathbb{R})$ is defined by $x_t(\theta):=x(t+\theta)$, for all $\theta \in [-\tau, 0]$. Therefore, the initial condition for equation (\ref{eq11-ps}) is 
\begin{equation}\label{ICond}
x_{t_0}=\phi\, , \, \, \phi \in \mathcal{C}_0^+\, .
\end{equation}

A solution of the IVP (\ref{eq11-ps}) - (\ref{ICond}) is denoted by $x(t;t_0,\phi)$, but for convenience, we will denote it by $x(t)$. We assume the existence of a unique solution for equation (\ref{eq11-ps}). 
The proof of the following Proposition follows by the comparison principle

\begin{prop}\label{Prop001} If conditions \textbf{A1}-\textbf{A3} holds, then any solution $x$ of equation (\ref{eq11-ps}) with initial condition $\phi \in \mathcal{C}_0^+$, satisfies
$$0\leq \dfrac{\sum_{i=1}^{n}\underline{\beta_i}\, \underline{f_i} +\underline{b} \, \underline{H}}{\overline{\alpha}}\leq \liminf_{t\to +\infty}x(t) \leq \limsup_{t\to +\infty}x(t) \leq \dfrac{\sum_{i=1}^{n}\overline{\beta_i}\, \overline{f_i} +\overline{b} \, \overline{H}}{\underline{\alpha}}\, .$$
\end{prop}
In particular, Proposition \ref{Prop001} means that all solutions of equation (\ref{eq11-ps}) are biological meaningful.

\noindent We state the following additional assumptions:
\begin{itemize}
\item[\bf A4] For all $1\leq i \leq n$, the functions $f_i:\R^+_0 \To \R^+_0$ are Lipschitz continuous, non-negative and they yield its maximum
value over $\R^+_0$, i.e. $\overline{f_i}=f_i(m_i^*)$ with
$m^*_i\in \R^+_0$ and $f_i$ is non-increasing function for $x > m^*_i$.
\item[\bf A5] There exist two positive constants $\gamma_1$ and $\gamma_2$ such that:
\begin{align*}
&0\leq
\frac{\overline{m^*}}{\underline{\lambda}}<\gamma_1<\frac{1}{\overline{\alpha}}\pa{\sum_{i=1}^n\underline{\beta_i}\, 
f_i(\overline{\lambda} {\gamma_2})+\underline{b}\, \underline{H}}\, ,\\
& \frac{1}{\underline{\alpha}}\pa{\sum_{i=1}^n
\overline{\beta_i}\, \overline{f_i}+ \overline{b}\, \overline{H} } <
\gamma_2\, ;
\end{align*}
where, $\displaystyle\overline{m^*}= \max_{1\leq 1\leq
n}\set{m^*_i}$.
\item[\bf A6] For all $1\leq i \leq n$, there exist positive numbers $\ell_{f_i}$ which are the Lipschitz constants of $f_i$ on $[\overline{m^*}, +\infty)$.
\end{itemize}

\begin{lem}\label{PermSol}
Let  $\Omega_0:=\{ \phi\,   \in\mathcal{C}^+\, :\,  \, \gamma_1 <\phi(t)< \gamma_2\, , \, t \in [-\tau,0] \}$. If conditions \textbf{A1}-\textbf{A5} are fulfilled; then, for every $\phi \in \Omega_0$, the solution $x(t)$ of equation (\ref{eq11-ps}) satisfies
$$\gamma_1 < x(t) < \gamma_2\, ,\ \, t \in [t_0, \zeta(\phi))\, $$
and its existence interval can be extended to $[t_0, +\infty)$.
\end{lem}
\begin{proof}
Firstly, let us proof that $x(t) < \gamma_2\, ,\ \, t \in [t_0, \zeta(\phi))$. If it were not true, then there must be $t_1 \in [t_0, \zeta(\phi))$ such that
$$ x(t_1)=\gamma_2,\,   {\rm and}\, \, x(t)<\gamma_2\, , \, \, \forall t\in [t_0-\tau,t_1)\, ;$$
then, using \textbf{A5}, we have the following contradiction
\begin{eqnarray*}
0 &\leq & x'(t_1)=-\alpha(t_1)u(t_1)+\sum_{i=1}^n\beta_i(t_1)f_i(\lambda_i(t_1)u(t_1-\tau_i(t_1)))+b(t_1)H(u(t_1-\sigma(t_1))\\
&<&-\underline{\alpha} \gamma_2 +\sum_{i=1}^n
\overline{\beta_i}\, \overline{f_i}+ \overline{b}\, \overline{H} \\
&<& 0\, .
\end{eqnarray*}
Similarly, suppose that the inequality $\gamma_1 < x(t) ,\ \, t \in [t_0, \zeta(\phi))\, $ does not holds; then,  there exists $t_2 \in [t_0, \zeta(\phi))$ such that
$$ x(t_2)=\gamma_1,\,   {\rm and}\, \, x(t)>\gamma_1\, , \, \, \forall t\in [t_0-\tau,t_2)\, ;$$
then, using \textbf{A5}, we have the following contradiction
\begin{eqnarray*}
0 &\geq & x'(t_2)=-\alpha(t_2)u(t_2)+\sum_{i=1}^n\beta_i(t_2)f_i(\lambda_i(t_2)u(t_2-\tau_i(t_2)))+b(t_2)H(u(t_2-\sigma(t_2))\\
&>&  -\overline{\alpha} \gamma_1 +\sum_{i=1}^n\underline{\beta_i}\, 
f_i(\overline{\lambda} {\gamma_2})+\underline{b}\, \underline{H} \\
&>& 0\, .
\end{eqnarray*}
The fact that $\zeta(\phi)=+\infty$, is a consequence of Theorem 3.2 in \cite{smith2011introduction}.
\end{proof}


\begin{thm}\label{thm3ps} Assume {\bf A1}-{\bf A6} hold. If
$$ \overline{b}L_H+\displaystyle\sum_{i=1}^n\overline{\beta_i}\, \overline{\lambda_i}\, \ell_{f_i}<\underline{\alpha} $$
holds, then equation (\ref{eq11-ps}) has a unique compact almost automorphic solution in the set
$$\Omega =\set{u\in KAA(\R, \R^+)\, : \,  \gamma_1\leq u(t)\leq \gamma_2, \, \, \,  \forall\, 
t\in\R}\, .$$
\end{thm}
\begin{proof}
Let us define the operator
\begin{equation}\label{eq15-ps}
({\mathfrak{M}}u)(t):=\int_{-\infty}^te^{-\int_s^t\alpha(\xi)d\xi}(Nu)(s)ds\, ,
\end{equation}
where, 
$N:C(\R)\To C(\R)$ is given by
\begin{equation}\label{eq12-ps}
N(u)(s)= \sum_{i=1}^n\beta_i(s)f_i(\lambda_i(s)u(s-\tau_i(s)))+
b(s)H(u(s-\sigma(s)))\, .
\end{equation}
Now we prove that $\Omega$  is closed, ${\mathfrak{M}}\left( \Omega \right) \subset \Omega$ and ${\mathfrak{M}} $ is a contraction; after that, the conclusion will follow from the Banach fixed point Theorem.

\begin{itemize}
\item $\Omega$ \textbf{is closed}. In fact, let the sequence $\{u_n\} \subset \Omega$ such that $u_n \to u$ uniformly on $\mathbb{R}$. Then, given $\epsilon >0,$ there exist $N_0 \in \mathbb{N}$ such that
$$||u_n-u||_{\infty}<\epsilon\, , \, \, \forall n\geq N_0\, .$$
Therefore, for all $t \in \mathbb{R}$ and $\forall n\geq N_0$, we have
$$|u(t)| \leq ||u_n-u||_{\infty}+|u_n(t)| < \epsilon+\gamma_2\, ,$$
$$\gamma_1 -\epsilon <u_n(t) -\epsilon < u(t)\, ;$$
which means that $u \in \Omega$.

\item ${\mathfrak{M}}\left( \Omega \right) \subset \Omega$. If $u$ is compact almost automorphic, then from the properties in Proposition \ref{PropCAA}, we have that $Nu$ is compact almost automorphic and from Proposition \ref{PropCBiaa} and Theorem \ref{InvConv1}, we have that ${\mathfrak{M}}(u)$ is compact almost automorphic too.  Moreover, if $u \in \Omega$, we have:
$$ \sum_{i=1}^n\underline{\beta_i}\, 
f_i(\overline{\lambda} {\gamma_2})+\underline{b}\, \underline{H}  \leq  \left(  Nu \right) (s) \leq   \sum_{i=1}^n
\overline{\beta_i}\overline{f_i}+ \overline{b}\, \overline{H}  \, ;$$
therefore,
$$\gamma_1< \dfrac{1}{\overline{\alpha}}\pa{ \sum_{i=1}^n\underline{\beta_i}\, 
f_i(\overline{\lambda} {\gamma_2})+\underline{b}\, \underline{H} } \leq \int_{-\infty}^te^{-\int_s^t\alpha(\xi)d\xi} \left(  Nu \right)(s)ds \leq \dfrac{1}{\underline{\alpha}}\pa{ \sum_{i=1}^n
\overline{\beta_i}\overline{f_i}+ \overline{b}\, \overline{H} } <\gamma_2\, . $$

\item ${\mathfrak{M}} $ \textbf{ is a contraction}. This follows from the inequality
$$||{\mathfrak{M}}(u)-{\mathfrak{M}}(v)||_{\infty} \leq \kappa ||u-v||_{\infty} 
\, ,\, \, \forall u,v \in \Omega\, ;$$
where,
$$\kappa:=\dfrac{1}{\underline{\alpha} }\pa{\overline{b}L_H+\displaystyle\sum_{i=1}^n\overline{\beta_i}\, \overline{\lambda_i}\, \ell_{f_i}} <1\, .$$
\end{itemize}
\end{proof}

\bmhead{Acknowledgements}

The authors wish to thank the anonymous reviewers for their helpful comments and suggestions, which have considerably improved the version of this work.


\section*{Declarations}


\begin{itemize}
\item Funding: A. Ch\'avez, N. Aragon\'es, U. Zavaleta and M. Pinto have been partially supported by research grand 038-2021-FONDECYT-Per\'u. A. Ch\'avez and M. Pinto have also been partially supported by CONCYTEC through the PROCIENCIA program under the E041-2023-01

\item Conflict of interest: The authors declare no conflict of interest.
\item Data availability: Not applicable.
\item Author contribution: The results were obtained by A.Ch., N.A., M.P. and U.Z. during many discussions, the authors have participated equally. Typing was done mostly by A.Ch. All authors corrected the manuscript.
\end{itemize}

\end{document}